\documentclass[10pt,reqno]{amsart}
\usepackage{amssymb, hyperref}
\def\wt#1{\widetilde{#1}}
\def\wh#1{\widehat{#1}}

\newcommand{\R}{{\mathbb R}}

\newcommand{\supp}{{\mbox{supp}}}

\newcommand{\airy}{{e^{-t\partial_x^3}}}

\newtheorem{theorem}{Theorem}
\newtheorem{proposition}{Proposition}
\newtheorem{lemma}{Lemma}
\newtheorem{corollary}{Corollary}
\newtheorem{remark}{Remark}

\newcommand{\px}{\partial_x}

\begin{document}
\title[Bilinear local smoothing estimate]{Bilinear local smoothing estimate for Airy equation}

\author{Soonsik Kwon}\address{Department of Mathematical Sciences\\
Korea Advanced Institute of Science and Technology\\
335 Gwahangno \\ Yuseong-gu, Daejeon 305-701, Republic of Korea}
\email{soonsikk@kaist.edu}

\author{Tristan Roy}\address{Courant Institute of Mathematical Sciences
New York University
251 Mercer Street
New York, NY 10012-1185, USA}\email{troy@cims.nyu.edu}

\begin{abstract}
In this short note, we prove a refinement of bilinear local smoothing estimate to Airy solutions, when the frequency support of two wave are separated. As an application we prove a smoothing property of a bilinear form.

\end{abstract}

\thanks{}
\thanks{} \subjclass[2000]{35Q53} \keywords{bilinear local smoothing estimate, Airy equation}

\maketitle

\noindent The purpose of this short note is to show a refinement of bilinear local smoothing estimate for Airy solutions:
\begin{equation}\label{airy}
\partial_t u + \partial_x^3 u =0. 
\end{equation} 
We denote $\airy u_0$ the linear solution, i.e.
$$ \airy u_0 = \frac{1}{2\pi}\int e^{it\xi^3+i(x-y)\xi} u_0(y)\, dyd\xi.  $$
The Airy equation is the linear part of the generalised Korteweg-de Vries equation: 
\begin{equation}\label{KdV}
\partial_t u + \partial_x^3 u + \partial_x(u^p) =0. 
\end{equation}
Local smoothing phenomenon is due to dispersive nature of linear dispersive equation, and firstly formulated by Kato \cite{Kato}. As like KdV equation, when the nonlinear term has a derivative the (local) smoothing estimate is crucial to build the Picard iteration in the well-posedness theory. Indeed, Kenig-Ponce-Vega \cite{KPV93} developed and used the following local smoothing estimate to obtain the well-posedness. 
\begin{proposition}
We have
 \begin{equation}\label{Eq:local smoothing}
  \|D^\alpha \airy f \|_{L^q_xL^r_t} \lesssim \|f\|_{L^2_x}
 \end{equation}
where $ -\alpha + \frac 1q + \frac 3r =\frac 12$, and $\frac 4q +\frac 2r \le 1 $, except at an end point
 $(q,r)=(\infty,\infty) $. Here $D^\alpha$ is a homogeneous fractional derivative. See Notations.\\
In particular, \begin{equation}\label{Eq:L5L10} \|\airy u_0 \|_{L^5_xL^{10}_t}  \lesssim \|u_0\|_{L^2_x}.  \end{equation}
Moreover, we have the inhomogeneous local smoothing estimate.
\begin{equation}\label{inhomogeneous local smoothing}
  \| \int e^{s \px^3} D^\alpha F(s,x)\, ds  \|_{L^2_x} \lesssim \|F \|_{L^{q'}_xL^{r'}_t}
\end{equation}
where $\frac 1q+ \frac 1{q'} =1 $ and $\frac 1r +\frac 1{r'} =1$ where $\alpha, q$ and $r$ as above.

\end{proposition}
\vspace{5mm}
\noindent In this note, we consider the interaction of two Airy wave when the support of frequencies are separated and show an improved version of bilinear local smoothing estimate.
\begin{theorem}\label{Th:bilinear}
Let $M,N >0$. Then,
  \begin{equation}\label{Eq:bilinear local smoothing1}
  \|D^\alpha \airy f D^\alpha \airy g \|_{L^{q/2}_x L^{r/2}_t} \lesssim \left(\frac{M}{N}\right)^{5 \theta/ 12} \| f \|_{L^2_x} \|g\|_{L^2_x}
  \end{equation}
  where
  $$ -\alpha +\frac 1q +\frac 3r =\frac 12, (\frac 1q,\frac 1r) = (\frac \theta6 +\frac{1-\theta}{4}, \frac \theta6), 0\le \theta \le 1 $$
 for all $L^2$-functions $f$ and $g$ with $ \text{supp} \, \widehat{f} \subset \{\xi:  |\xi| \le 2M \}$ and $ \text{supp \,} \widehat{g} \subset \{\xi: N \le |\xi|\le 2N \}, 0\le M\le N . $\\
\noindent In particular, we have
  \begin{equation}\label{Eq:bilinear local smoothing2}
    \|\airy f \airy g \|_{L^{5/2}_xL^{5}_t} \lesssim \left(\frac{M}{N}\right)^{1/4}\|f\|_{L^2_x}\|g\|_{L^2_x}
  \end{equation}
\end{theorem}
\vspace{5mm}
\noindent In the space-time frequency space, the linear wave is supported on the characteristic curve $\tau=\xi^3$. Due to the curvature (or the slope of the tangent line) of the interaction of two linear waves at different frequencies is weaker and so one can have some gain.
\begin{remark}
This type of estimate is firstly shown by Bourgain for symmetric Strichartz estimate of Schr\"odinger equation in $d=2$ \cite{Bourgain98}.
Keraani-Vargas \cite{Keraani-Vargas} extended to other dimension for symmetric Strichartz norms and Chae-Cho-Lee \cite{Chae-Cho-Lee} for
non-symmetric norms.
\end{remark}
\begin{remark}\label{Re:counter example}
  The exponent in Theorem~\ref{Th:bilinear} is sharp. \\
  Define $\widehat{f} = \chi_{M\le \xi \le 2M}$ and $ \widehat{g}=\chi_{1 \le \xi \le 1+M^{1/2}}$. Consider a subset $K$ of $\R\times\R$ of $(t,x)$
  $$ K=\{(t,x) : |x+3t| \le \frac{1}{100}M^{-1/2}, |x|\le \frac{1}{100}M^{-1} \}. $$
  One can easily observe that for all $(t,x) \in K$,
  $$ | \airy f(x)|=|\int_M^{2M} e^{it\xi^3+ix\xi} d\xi| \sim M$$
  and $$ | \airy g(x)| =|\int_{1}^{1+M^{1/2}} e^{it\xi^3+ix\xi} d\xi| \sim M^{1/2}. $$
  Thus,
  \begin{align*}
  \|D^\alpha \airy f A^\alpha \airy g \|_{L^{\frac q2}_xL^{\frac r2}_t} &\ge M^{1}M^{\frac 12}M^\alpha \|\chi_{K} \|_{L^{q/2}_xL^{r/2}_t} \\
                         & \sim M^{\frac 32+\alpha}M^{-\frac 12\cdot \frac 2r}M^{-1\cdot \frac 2q} \\
                         & \sim M^{1+\frac 2r - \frac 1q} =M^{\frac{5\theta}{12} +\frac 34}
  \end{align*}
  where used admissible condition of exponents. \\
Since $\|f\|_{L^2_x} =M^{1/2}$ and $\|g\|_{L^2_x}=M^{1/4}$, we see the estimate \eqref{Eq:bilinear local smoothing1} is sharp.

\end{remark}
\vspace{5mm}
\begin{proof}[Proof of Theorem \ref{Th:bilinear}]
\noindent \newline Since \eqref{Eq:bilinear local smoothing1} is a scaling invariant estimate, by scaling one can assume $N=1$. In view of
\eqref{Eq:local smoothing}, we may also assume $M \ll 1$. \eqref{Eq:bilinear local smoothing1} follows by interpolating the following two
estimates:
\begin{equation}\label{Eq:bilinear 1}
  \|D^{-1/4}\airy fD^{-1/4} \airy g \|_{L^2_xL^\infty_t} \lesssim \|f\|_{L^2_x}\|g\|_{L^2_x}
\end{equation}
\begin{equation}\label{Eq:bilinear 2}
 \|D^{1/6}\airy f D^{1/6} \airy g\|_{L^3_{x,t}} \lesssim M^{5/12}\|f\|_{L^2_x}\|g\|_{L^2_x}
\end{equation}
\eqref{Eq:bilinear 1} is an immediate result of \eqref{Eq:local smoothing}. Now we prove \eqref{Eq:bilinear 2}. Using Bernstein's inequality and
observing frequency support of $f$ and $g$, we are reduced to show that
\begin{equation}\label{Eq:bilinear3}
\|\airy f \airy g\|_{L^3_{x,t}} \lesssim M^{1/4}\|f\|_{L^2_x}\|g\|_{L^2_x}.
\end{equation}
\eqref{Eq:bilinear3} is derived from the following lemma: indeed, it follows from the interpolation of
\eqref{Eq:hausdorff-young}(p=2) and \eqref{Eq:bilinear4}

\end{proof}

\begin{lemma}\label{Le:lemma 1}
  Assume that $f$ and $g$ are functions such that $ \text{supp} |\widehat{f}| \subset [0,2M], \text{supp} |\widehat{g}| \subset [1,2]$, $ M \ll 1$  \\
 (a) Let $p\ge2$. Then we have
  \begin{equation}\label{Eq:hausdorff-young}
    \|\airy f\airy g \|_{L^p_{x,t}} \lesssim \|\widehat{f}\|_{L^{p'}_\xi}\|\widehat{g}\|_{L^{p'}_\xi}
  \end{equation}
  where $p'= \frac{p}{p-1}.$ \\
 (b) \begin{equation}\label{Eq:bilinear4}
   \|\airy f\airy g\|_{L^4_{x,t}} \lesssim M^{3/8}\|f\|_{L^2_x}\|g\|_{L^2_x}
 \end{equation}
\end{lemma}
\vspace{5mm}
\noindent  Lemma \ref{Le:lemma 1}(a) follows from a classical argument of Fefferman and Stein \cite{Fefferman}. It makes use of Hausdorff-Young inequality. We give a proof in the appendix for the sake of completeness.\\
In order to show \eqref{Eq:bilinear4} we use the example of Remark \ref{Re:counter example}. Decompose $g$ into functions whose
frequency support are on small intervals of length $M^{1/2}$. Indeed, for integer $k$, $ 1\le k \le M^{-1/2}$, set $I_k=[ 1+ (k-1)M^{1/2}, 1+
kM^{1/2}]$ and set $\widehat{g_k} =\widehat{g}\chi_{I_k} $. Then $g=\sum_k g_k$. Then we will use the following orthogonality inequality:
\begin{lemma}\label{Le:orthogonality}
We have
\begin{equation}\label{Eq:orthogonality}
\|\sum_k \airy f\airy g_k \|_{L^4_{t,x}} \lesssim \Big(\sum_k \|\airy f \airy g_k \|^2_{L^4_{t,x}}\Big)^{1/2}
\end{equation}
\end{lemma}


\begin{proof} $ $ \\
We write, using Plancherel theorem,
\begin{align*}
  \|\sum_k \airy f\airy g_k\|^2_{L^4_{x,t}} &= \|(\sum_k \airy f \airy g_k)^2\|_{L^2_{t,x}}\\
          &=\|\sum_k \airy f \airy g_k\sum_j\airy f\airy g_j\|_{L^2_{t,x}} \\
          &= \|\sum_{j,k} \widetilde{\airy f}\ast \widetilde{\airy f}\ast \widetilde{\airy g_j}\ast\widetilde{\airy g_k} \|_{L^2_{\tau,\xi}}.
\end{align*}
where $\wt{f(t,x)}(\tau,\xi)$ is the space-time Fourier transform of $f(t,x)$.
We denote by $E_{j,k}$ the support of the function $\widetilde{\airy f}\ast \widetilde{\airy f}\ast \widetilde{\airy g_j}\ast\widetilde{\airy
g_k}. $ We claim that the $E_{j,k}$ are essentially disjoint. In other words, there is a constant $C$, independent of $M$, so
that
\begin{equation}\label{Eq:disjoint}
 \sum_{j,k} \chi_{E_{j,k}} \le C.
 \end{equation}
By this claim, we estimate
\begin{align*}
  \|&\sum_{j,k} \widetilde{\airy f}\ast \widetilde{\airy f}\ast \widetilde{\airy g_j}\ast\widetilde{\airy g_k} \|_{L^2_{t,x}} \\
    & \le C \Big(\sum_{j,k}\|\widetilde{\airy f}\ast \widetilde{\airy f}\ast \widetilde{\airy g_j}\ast\widetilde{\airy g_k}  \|^2_{L^2_{\tau,\xi}}\Big)^{1/2} \\
   & = C \Big(\sum_{j,k}\| \airy f\airy g_k\airy f \airy g_j\|_{L^2_{\tau,\xi}} \Big)^{1/2} \\
   & = C  \Big(\sum_{j,k} \int |\airy f\airy g_k\airy f \airy g_j |^2 \Big)^{1/2} \\
   & = C  \Big(\int (\sum_{k} |\airy f \airy g_k|^2)^2 \Big)^{1/2} \\
    & = C\|\sum_k|\airy f \airy g_k|^2 \|_{L^2} \\
    & \le C \sum_k \| |\airy f\airy g|^2 \|_{L^2} \\
    & = C\sum_k\|\airy f\airy g_k\|^2_{L^4}.
\end{align*}
We are left to show the inequality \eqref{Eq:disjoint}. One can easily see that the support of $\widetilde{\airy g_k}$ is in
$E_k=\{(\tau,\xi): |\xi-kM^{1/2}| \le M^{1/2}, \tau =\xi^3 \}$, and the support of $\widetilde{\airy f} $ is in $\{(\tau,\xi): |\xi| \le 2M,
\tau = \xi^3 \}$. If $(\rho,\eta) \in E_{j,k} $, then there exists $(\xi_1,\xi_2)$ such that $(\xi_1^3,\xi_1) \in E_k, (\xi_2^3,\xi_2)\in E_j, |\rho - \xi_1^3 -\xi_2^3 | \le 4M,$ and $|\eta - \xi_1 -\xi_2|\le 4M$. \\
From the identity $ 4\xi_1^3 +4\xi_2^3 = (\xi_1 +\xi_2)^3 + 3(\xi_1-\xi_2)^2(\xi_1+\xi_2) $, we see that
$$ E_{j,k} \subset F_{j,k} = \{(\rho,\eta): |\eta -(j+k)M^{1/2}| \le 3M^{1/2}, (3|k-j|^2-8)M \le |4\rho -\eta^3| \le (6|k-j|^2+8)M \}. $$
It is easy to verify that the $F_{j,k}$'s overlap only a finite number of times and that this number is bounded by a universal constant.

\end{proof}
\begin{proof}[Proof of \eqref{Eq:bilinear4}]
$  $ \\
In view of \eqref{Eq:orthogonality}, we are reduced to show
$$\Big(\sum_k \|\airy f \airy g_k \|^2_{L^4_{t,x}}\Big)^{1/2} \lesssim M^{3/8}\|f\|_{L^2_x}\|g\|_{L^2_x}.$$
Using \eqref{Eq:hausdorff-young} for $p=4$, and the size of support of $f$ and $g_k$, we estimate
\begin{align*}
  \Big(\sum_k \|\airy f \airy g_k \|^2_{L^4_{t,x}}\Big)^{1/2} &\lesssim \Big(\sum_k \|\widehat{f}\|^2_{L^{4/3}}\|\widehat{g_k}\|^2_{L^{4/3}}\Big)^{1/2} \\
           &\lesssim M^{3/8}\|\widehat{f}\|_{L^2}\big(\sum_k \|\widehat{g_k}\|^2_{L^2}\big)^{1/2} \\
           & = M^{3/8}\|f\|_{L^2} \|g\|_{L^2}
\end{align*}
which concludes \eqref{Eq:bilinear4}.\\

\end{proof}

\noindent As a corollary we can observe the smoothing property of bilinear form. 

\begin{corollary}
  Let $\frac 34 <b < 1$.
  \begin{equation}\label{Eq:bilinear derivative}
  \|\px (\airy u_0 \airy v_0) \|_{L^{5/2}_xL^5_t} \lesssim \|u_0\|_{H^b}\|v_0\|_{H^{1-b}}
  \end{equation}
  \label{cor:BilDer}
\end{corollary}
\noindent We have a $\frac 14 -\epsilon $ derivative gain in the bilinear form compared when used \eqref{Eq:local smoothing}. The proof is via the frequency decomposition. In the high-low interaction when a derivative falls to high frequency, the refined bilinear estimate helps to move some derivative to low frequency part. We add the proof in the Appendix. 

\vspace{5mm}

\section*{Appendix}
\begin{proof}[Proof of \eqref{Eq:hausdorff-young}]
\noindent \newline Writing
$$ \airy f \airy g(t,x) = c\iint e^{ix(\xi_1+\xi_2) + it(\xi_1^3 +\xi_2^3)} \wh{f}(\xi_1)\wh{g}(\xi_2)\, d\xi_1d\xi_2, $$
we make a change of variables $(u,v) = (\xi_1+\xi_2, \xi_1^3+\xi_2^3).$ Then we obtain
$$ \airy f \airy g(t,x) = c\iint e^{ixu+itv} \Pi(u,v) \,dudv $$
where $\Pi(u,v) = \wh{f}(\xi_1)\wh{g}(\xi_2) |J^{-1}|$ and $ J=\det \frac{\partial(u,v)}{\partial(\xi_1,\xi_2)}= \frac{1}{3(\xi_2^2-\xi_1^2)}.$ \\
We can view
$$ \airy f \airy g(t,x) = \wh{\Pi}(t,x). $$ Hence, using Hausdorff-Young inequality, for $p\ge 2$,
$$ \|\airy f \airy \|_{L^p_{t,x}}  = \|\wh{\Pi} \|_{L^p_{t,x}} \le \|\Pi \|_{L^{p'}_{t,x}}  $$
where $p'=\frac{p}{p-1}.$ \\
To compute $\|\Pi\|_{L^{p'}}$, we use the fact $ |\xi_1-\xi_2| \ge 1/2 $ (i.e. $|J| \sim 1$) and change variables back to $\xi_1,\xi_2$. Indeed,
\begin{align*}
  \|\Pi\|^{p'}_{L^{p'}} &= \iint \big|\wh{f}(\xi_1)\wh{g}(\xi_2) J^{-1} \big|^{p'} \,dudv \\
                        &= \iint \big|\wh{f}(\xi_1)\wh{g}(\xi_2) \big|^{p'} \big| J^{-1}\big|^{p'}\big|J \big| \,d\xi_1 d\xi_2 \\
                        &\sim \|\wh{f}\|^{p'}_{L^{p'}} \|\wh{g}\|^{p'}_{L^{p'}}.
\end{align*}
\end{proof}

\begin{proof}[Proof of Corollary \ref{cor:BilDer}]
\noindent \newline   We use the Littlewood-Paley operators to decompose into the paraproduct:
  $$ \px \airy u_0 \airy v_0 = \pi_{lh} +\pi_{hh} +\pi_{hl} $$
  where
  \begin{align*}
    \pi_{lh} &= \sum_{N<M} P_N\px \airy u_0 P_M \airy v_0  \\
    \pi_{hh} &= \sum_{N \sim M } P_N\px \airy u_0 P_M \airy v_0 \\
    \pi_{hl} &= \sum_{N>M} P_N\px \airy u_0 P_M \airy v_0.
  \end{align*}
By the triangle inequality, we have
\begin{equation*}
 \|\px (\airy u_0 \airy v_0) \|_{L^{5/2}_xL^5_t} \le \|\pi_{lh}\|_{L^{5/2}_xL^5_t}+  \|\pi_{hh}\|_{L^{5/2}_xL^5_t}+  \|\pi_{hl}\|_{L^{5/2}_xL^5_t}.
\end{equation*}
We estimate term by term. For the first two terms we can use the usual local smoothing estimate since the derivative falls in the low frequency part.
\begin{align*}
\|\pi_{hh}\|_{L^{5/2}_xL^5_t} &\lesssim \sum_{j=-1}^{\infty} \|P_{2^j}\px \airy u_0 P_{2^j}\airy v_0\|_{L^{5/2}_xL^5_t} \\
    &\lesssim \sum_{j=-1}^{\infty} \|\widetilde{P}_{2^j}\px\airy u_0 \|_{L^{5}_xL^{10}_t} \|\widetilde{P}_{2^j}\airy v_0 \|_{L^{5}_xL^{10}_t}  \\
    &\lesssim  \sum_{j=-1}^{\infty} 2^{j}\|\widetilde{P}_{2^j}\airy u_0 \|_{L^{5}_xL^{10}_t} \|\widetilde{P}_{2^j}\airy v_0 \|_{L^{5}_xL^{10}_t}  \\
    &=  \sum_{j=-1}^{\infty}  2^{bj}\|\widetilde{P}_{2^j}\airy u_0 \|_{L^{5}_xL^{10}_t} 2^{j(1-b)}  \|\widetilde{P}_{2^j}\airy v_0 \|_{L^{5}_xL^{10}_t}  \\
    &\lesssim  \sum_{j=-1}^{\infty}  2^{bj}\|\widetilde{P}_{2^j} u_0 \|_{L^{2}} 2^{j(1-b)}  \|\widetilde{P}_{2^j} v_0 \|_{L^2_x} \\
    &\lesssim  \|u_0\|_{H^b} \|v_0\|_{H^{1-b}}
\end{align*}
where $ \widetilde{P}_{2^j} = P_{2^{j-1}} + P_{2^j} +P_{2^{j+1}} $.
\begin{align*}
\|\pi_{lh}\|_{L^{5/2}_xL^5_t} &\lesssim \sum_{j=1}^{\infty} \sum_{k=-1}^{j-1} \|P_{2^k}\px \airy u_0 P_{2^j}\airy v_0\|_{L^{5/2}_xL^5_t} \\
    &\lesssim \sum_{j=1}^{\infty} \sum_{k=-1}^{j-1} \|\widetilde{P}_{2^k}\px\airy u_0 \|_{L^{5}_xL^{10}_t} \|\widetilde{P}_{2^j}\airy v_0 \|_{L^{5}_xL^{10}_t}  \\
    &\lesssim  \sum_{j=1}^{\infty} \sum_{k=-1}^{j-1}  2^{k}\|\widetilde{P}_{2^k}\airy u_0 \|_{L^{5}_xL^{10}_t} \|\widetilde{P}_{2^j}\airy v_0 \|_{L^{5}_xL^{10}_t}  \\
    &=  \sum_{j=1}^{\infty} \sum_{k=-1}^{j-1} 2^{(1-b)(k-j)} 2^{bk}\|\widetilde{P}_{2^k}\airy u_0 \|_{L^{5}_xL^{10}_t} 2^{j(1-b)}  \|\widetilde{P}_{2^j}\airy v_0 \|_{L^{5}_xL^{10}_t}  \\
    &\lesssim  \sum_{j=1}^{\infty} \sum_{k=-1}^{j-1} 2^{(1-b)(k-j)} 2^{bk}\|\widetilde{P}_{2^k} u_0 \|_{L^{2}} 2^{j(1-b)}  \|\widetilde{P}_{2^j} v_0 \|_{L^2_x} \\
    &\lesssim \sum_{i=1}^{\infty} 2^{-(1-b)i} \sum_{j \ge i} 2^{b(j-i)}\|\widetilde{P}_{2^{j-i}} u_0\|_{L^2} 2^{(1-b)j} \|\widetilde{P}_{2^{j}} v_0 \|_{L^2} \\
    &\lesssim  \|u_0\|_{H^b} \|v_0\|_{H^{1-b}}.
\end{align*}
For the last term, the high-low paraproduct, we need to use improved bilinear estimate \eqref{Eq:bilinear local smoothing2}.
\begin{align*}
\|\pi_{hl}\|_{L^{5/2}_xL^5_t} &\lesssim \sum_{j=1}^\infty \sum_{k=-1}^{j-1} \|P_{2^j}\airy \px u_0 P_{2^k} \airy v_0 \|_{L^{5/2}_xL^5_t}  \\
                            &\lesssim \sum_{j=1}^\infty \sum_{k=-1}^{j-1} 2^{(k-j)/4} 2^{j}\|P_{2^j} u_0\|_{L^2} \|P_{2^k} v_0\|_{L^2} \\
                            &= \sum_{j=1}^\infty \sum_{k=-1}^{j-1} 2^{(k-j)(b-3/4)} 2^{jb}\|P_{2^j} u_0\|_{L^2}2^{kb}\|P_{2^k} v_0\|_{L^2} \\
                            &\lesssim_b \|u_0\|_{H^b}\|v_0\|_{H^{1-b}}
\end{align*}
where we used Bernstein's inequality, Cauchy-Schwartz inequality, and \eqref{Eq:bilinear local
smoothing2}.
\end{proof}

\subsection*{Notations}
We use space-time mixed norm notation:
$$ \|u(t,x)\|_{L^q_xL^r_t} := \Big(\int \big(\int |u(t,x)|^r dt\big)^{q/r} dx \Big)^{1/q}. $$
We denote the fractional derivative as $ \widehat{D^s f}(\xi) = |\xi|^s \widehat{f}(\xi) $ and the Sobolev norm as
$$  \|f\|_{H^s} = \| \langle D\rangle^s f \|_{L^2} $$
where $\langle \xi\rangle= |\xi|+1$ and $\widehat{f}$ is the Fourier transform of $f$.
We use $X \lesssim Y$  to denote the estimate $X \le C Y$ where $C$ depends only on the fixed parameters and exponents. We shall need the
following Littlewood-Paley projection operators. Let $\phi(\xi)$ be a bump function, $\supp\, \phi \in \{ |\xi| \le 2\} $ and $\phi(\xi) =1 $ on $
\{ |\xi| \le 1\}$. For each dyadic number $N=2^j, j \in \mathbb{N}$,
\begin{align*}
 \wh{P_N f}(\xi) = (\phi(\xi/N) - \phi(2\xi/N)) \wh{f}(\xi)  \\
 \wh{P_0 f}(\xi) = (\phi(\xi))\wh{f}(\xi) \\
 \wh{P_{\le N} f}(\xi) = \sum_{M\le N} P_M f(\xi) \\
 \wh{P_{>N} f}(\xi) = \wh{(I-P_{\le N}) f}(\xi)
\end{align*}
and we also use a wider projection operator $\wt{P}_N= P_{N/2} + P_N + P_{2N}$.

\end{document}